\pdfoutput=1
\documentclass[11pt,a4paper,reqno]{amsart}

\usepackage[british]{babel}

\usepackage{microtype}
\usepackage[T1]{fontenc}
\usepackage{setspace}
\usepackage{amssymb}
\usepackage{amsmath}
\usepackage{amsthm}
\usepackage{enumitem}
\usepackage{mathrsfs}
\usepackage[colorlinks,citecolor=blue,urlcolor=black,linkcolor=blue,linktocpage]{hyperref}
\usepackage{mathtools}
\usepackage{xcolor}
\usepackage{xfrac}

\renewcommand{\epsilon}{\varepsilon}
\renewcommand{\phi}{\varphi}
\newcommand{\defeq}{\mathrel{\mathop:}=}

\renewcommand{\Re}{\operatorname{Re}}

\DeclareMathOperator{\Neigh}{\mathcal{U}}
\DeclareMathOperator{\PD}{P_{1}}
\DeclareMathOperator{\conv}{co}
\DeclareMathOperator{\cconv}{\overline{co}}
\DeclareMathOperator{\Ker}{Ker}
\DeclareMathOperator{\ch}{Ch}
\DeclareMathOperator{\rk}{rk}
\DeclareMathOperator{\rank}{rank}
\DeclareMathOperator{\M}{M}
\DeclareMathOperator{\lat}{L}
\DeclareMathOperator{\E}{E}
\DeclareMathOperator{\GL}{GL}
\DeclareMathOperator{\SL}{SL}
\DeclareMathOperator{\cent}{Z}
\DeclareMathOperator{\C}{\mathbb{C}}
\DeclareMathOperator{\R}{\mathbb{R}}
\DeclareMathOperator{\Q}{\mathbb{Q}}
\DeclareMathOperator{\Z}{\mathbb{Z}}
\DeclareMathOperator{\N}{\mathbb{N}}
\DeclareMathOperator{\T}{\mathbb{T}}
\DeclareMathOperator{\F}{\mathbb{F}}
\DeclareMathOperator{\diam}{diam}
\DeclareMathOperator{\U}{U}
\DeclareMathOperator{\A}{A}
\DeclareMathOperator{\Hom}{Hom}

\newtheorem{thm}{Theorem}[section]
\newtheorem{cor}[thm]{Corollary}
\newtheorem{lem}[thm]{Lemma}
\newtheorem{prop}[thm]{Proposition}
\newtheorem{problem}[thm]{Problem}

\theoremstyle{definition}

\newtheorem{remark}[thm]{Remark}

\begin{document}

\setlist{noitemsep}

\author{Friedrich Martin Schneider}
\address{F.M.S., Institute of Discrete Mathematics and Algebra, TU Bergakademie Freiberg, 09596 Freiberg, Germany}
\email{martin.schneider@math.tu-freiberg.de}
\thanks{F.M.S.~acknowledges funding by the Deutsche Forschungsgemeinschaft (DFG, German Research Foundation) -- Projektnummer 561178190}

\author{Andreas Thom}
\address{A.T., Institute of Geometry, TU Dresden, 01062 Dresden, Germany}
\email{andreas.thom@tu-dresden.de}

\title[Unitary representations and continuous geometries]{Unitary representations and von Neumann's continuous geometries}
\date{\today}

\begin{abstract} We prove that the unit group of a non-discrete irreducible, continuous ring, in the sense of John von Neumann, does not admit any non-trivial unitary representation continuous with respect to the strong operator topology. \end{abstract}



\maketitle

\allowdisplaybreaks

\tableofcontents

\section{Introduction}\label{section:introduction}

Since the advent of quantum mechanics, it has been natural to seek a mathematical framework for the logic of quantum propositions. This was John von Neumann's main original motivation for introducing \emph{continuous geometries}~\cite{VonNeumannBook}---a perspective that was prominently advocated in his influential work with Birkhoff~\cite{BvN1936}; see also~\cite{redei,KronzLupher}. Our main result shows, however, that the group of natural symmetries attached to a continuous geometry is inherently incompatible with the symmetries of Hilbert space. Thus, our result obstructs a direct implementation of von Neumann's original program of connecting the theory of continuous geometries with the usual Hilbert-space formalism of quantum theory.

Extending the classical Veblen--Young theorem, continuous geometries admit a distinguished coordinatization: they are, in a unique way, isomorphic to the lattices of principal right ideals of certain regular rings, the so-called \emph{continuous rings}~\cite{VonNeumannBook}. By another fundamental theorem of von Neumann (Theorem~\ref{theorem:rank.function}), every irreducible, continuous ring carries a unique rank function $\rk_{R}\colon R \to [0,1]$, and the associated metric $d_{R}\colon (a,b) \mapsto \rk_{R}(a-b)$ is complete. The resulting topology, the \emph{rank topology}, is discrete if and only if the ring is isomorphic to a finite-dimensional matrix algebra over some division ring. On the other hand, there exists a wealth of irreducible, continuous rings with non-discrete rank topology. The basic examples arise as rank-metric completions of inductive limits of matrix algebras over fields (Remark~\ref{remark:simple.example}). Furthermore, the class of irreducible, continuous rings is closed under metric ultraproducts---a fact which has found applications in the study of group algebras~\cite{ElekSzabo,linnell}. Another intriguing family of examples has its origin in the theory of operator algebras, namely in the work of Murray and von Neumann~\cite{MurrayVonNeumann}: the algebra of densely defined, closed, linear operators \emph{affiliated} with a $\mathrm{II}_{1}$ factor $M$ constitutes a non-discrete irreducible, continuous ring whose lattice of principal right ideals is naturally isomorphic to the projection lattice of $M$.

Studying the natural symmetries of a continuous geometry leads us to considering the unit group of its coordinate ring. More precisely, if $R$ is any irreducible, continuous ring, then the corresponding unit group \begin{displaymath}
    \GL(R) \, \defeq \, \{ a \in R \mid \exists b \in R \colon \, ab = ba = 1 \}
\end{displaymath} naturally inherits the rank topology from $R$. Indeed, since the metric $d_{R}$ restricts to a bi-invariant metric on $\GL(R)$, the latter becomes a topological group. It is this topological group that we study in the present work.

A \emph{continuous unitary representation} of a topological group $G$ on a Hilbert space~$H$ is a homomorphism from $G$ to the unitary group $\U(H)$ continuous with respect to the strong operator topology. While many topological groups, including all locally compact ones, admit a rich supply of such representations,  there do exist \emph{exotic} groups at the other end of the spectrum, i.e., topological groups not possessing any non-trivial continuous unitary representation: examples include $L^{0}$ groups over pathological submeasures~\cite{HererChristensen,SchneiderSolecki}, certain Banach--Lie groups~\cite{banaszczyk,banaszczyk2}, the group of
orientation-preserving homeomorphisms of the closed real unit interval~\cite{megrel}, the isometry group of the Urysohn space of diameter
one~\cite{Pestov07}, as well as
certain limits of finite special linear groups~\cite{CarderiThom}. Inspired by and extending the latter example, the present work identifies continuous geometries as a new source of this phenomenon.

\medskip

Our main result reads as follows.

\begin{thm}\label{theorem:exotic} Let $R$ be a non-discrete irreducible, continuous ring. Then the topological group $\GL(R)$ is exotic, i.e., it does not admit any non-trivial continuous unitary representations. \end{thm}

The proof of our Theorem~\ref{theorem:exotic} in finite characteristic is based on results by the second author and Carderi in \cite{CarderiThom}, which in turn are based on character estimates for unitary representations of finite simple groups by Gluck \cite{Gluck}. In case of characteristic zero we need a similarly strong input, which is the work of Bader--Boutonnet--Houdayer--Peterson \cite{BaderBoutonnetHoudayerPeterson} on charmenability of arithmetic groups, a vast generalization of Margulis' normal subgroup theorem~\cite{Margulis} and the Stuck--Zimmer stabilizer rigidity theorem~\cite{StuckZimmer}. Finally, we apply a character rigidity theorem due to Kirillov~\cite{Kirillov}, see also~\cite{PetersonThom}, as well as a topological simplicity result from the first author's recent work~\cite{Schneider25a}.

The historically first work on exoticness, due to Herer and Christensen~\cite{HererChristensen}, established this property as a possible ingredient for proving extreme amenability. A topological group $G$ is called \emph{extremely amenable} if every continuous action of $G$ on a non-empty compact Hausdorff space admits a $G$-fixed point. The topological groups considered in Theorem~\ref{theorem:exotic} are not always extremely amenable~\cite{SchneiderIMRN}, but they are known to be \emph{inert}~\cite[Corollary~1.6]{SchneiderGAFA}, i.e., whenever they act continuously on a non-empty compact Hausdorff space every single group element has a fixed point. Our Theorem~\ref{theorem:exotic} reveals another dynamical peculiarity of these objects. Recall that a topological group $G$ is said to be \emph{whirly}~\cite[Remark~5.7]{SchneiderSolecki} if, for every continuous action of $G$ on a compact Hausdorff space $X$ and every $G$-invariant regular Borel probability $\mu$ on $X$, the support of $\mu$ is contained in the set of $G$-fixed points in~$X$.

\begin{cor}\label{corollary:whirly} Let $R$ be a non-discrete irreducible, continuous ring. Then the topological group $\GL(R)$ is whirly. \end{cor}

Combined with a recent automatic continuity result from~\cite[Theorem~1.3]{BernardSchneider}, our Theorem~\ref{theorem:exotic} has non-trivial ramifications even for the underlying abstract groups (not carrying any topology).

\begin{cor}\label{corollary:separable.reps} Let $R$ be a non-discrete irreducible, continuous ring. Every unitary representation of the group $\GL(R)$ on a separable Hilbert space is trivial. \end{cor}

This note is organized as follows. After recollecting some general facts about positive-definite functions in Section~\ref{section:positive.definite.functions}, we provide the relevant background on continuous rings and their unit groups in Section~\ref{section:continuous.rings}. The subsequent Section~\ref{section:charmenability} contains the central technical ingredients about characters on special linear groups needed for the proof of our main results given in Section~\ref{section:main.result}. In the final Section~\ref{section:open.problem}, we suggest an open problem for future research.

\section{Dynamics of positive-definite functions}\label{section:positive.definite.functions}

In this section, we collect some general observations related to continuous positive-definite functions on topological groups.

To clarify some notation and terminology, let $G$ be a group. Then there is a natural right action of $G$ on $\C^{G}$ given by \begin{displaymath}
    \phi^{g} \colon \, G \, \longrightarrow \, \C, \quad h \, \longmapsto \, \phi\!\left(ghg^{-1}\right) \qquad \left( \phi \in \C^{G}, \, g \in G \right),
\end{displaymath} and we define \begin{displaymath}
    \left. \Delta(G) \, \defeq \, \left\{ \phi \in \C^{G} \, \right\vert \forall g \in G \colon \, \phi^{g} = \phi \right\} .
\end{displaymath} Furthermore, for $\phi \in \C^{G}$ and a subgroup $H \leq G$, we let $\phi^{H} \defeq \{ \phi^{h} \mid h \in H \}$. Given any $A \subseteq \C^{G}$, we consider $\cconv (A) \defeq \overline{\conv (A)}$, i.e., the topological closure of the convex hull of $A$ in $\C^{G}$ with respect to the product topology. A \emph{positive-definite} function on $G$ is a map $\phi \colon G \to \C$ such that \begin{displaymath}
	\sum\nolimits_{i,j=1}^{n} c_{i}\overline{c_{j}} \phi\!\left(g_{j}^{-1}g_{i}\right)\! \, \geq \, 0 
\end{displaymath} for all $n \in \N$, $c_{1},\ldots,c_{n} \in \C$ and $g_{1},\ldots,g_{n} \in G$. A \emph{character} on $G$ is a positive-definite element of $\Delta(G)$, i.e., a positive-definite function $\chi \colon G \to \C$ such that $\chi^{g} = \chi$ for all $g \in G$. The reader is referred to~\cite[Appendix~C.4]{KazhdanBook} for further background on positive-definite functions. We confine ourselves to the following basic facts.

\begin{lem}[{\cite[Proposition~C.4.2(ii),\,(iii)]{KazhdanBook}}]\label{lemma:positive.definite.functions} Let $G$ be a group and let $\phi \colon G \to \C$ be positive-definite. Then, for all $g,h \in G$, \begin{enumerate}
    \item\label{lemma:positive.definite.functions.1} $\vert \phi(g) \vert \leq \phi(1)$, 
	\item\label{lemma:positive.definite.functions.2} $\vert \phi(g) - \phi(h) \vert^{2} \leq 2\phi(1)\!\left(\phi(1)-\Re \phi\!\left(g^{-1}h\right)\right)$.
\end{enumerate} \end{lem}

For any topological group $G$, the set \begin{displaymath}
    \left. \PD(G) \, \defeq \, \left\{ \phi \in \C^{G} \, \right\vert \phi \text{ continuous, positive-definite, } \phi(1) = 1 \right\}
\end{displaymath} constitutes a $G$-invariant convex subset of $\C^{G}$.

\begin{lem}\label{lemma:exoticness.via.positive.definite.functions} Let $G$ be a topological group. Then \begin{displaymath}
    \bigcap \{ \Ker \pi \mid \pi \text{ continuous unitary representation of } G \} = \bigcap\nolimits_{\phi \in \PD(G)} \phi^{-1}(1).
\end{displaymath} In particular, $G$ is exotic if and only if $\PD(G) = \{ 1 \}$. \end{lem}

\begin{proof} The inclusion ($\subseteq$) follows from~\cite[Theorem~C.4.10]{KazhdanBook}, while the inclusion ($\supseteq$) is a consequence of~\cite[Proposition~C.4.3]{KazhdanBook}. \end{proof}

Given a topological group $G$, we let $\Neigh(G)$ be the neighborhood filter of the neutral element $1 \in G$. A topological group $G$ is said to have \emph{small invariant neighborhoods} \emph{(SIN)} if $\Neigh(G)$ admits a filter base consisting of conjugation-invariant sets, i.e., \begin{displaymath}
    \forall U \in \Neigh(G) \ \exists V \in \Neigh(G) \colon \qquad (V \subseteq U) \, \wedge \, \left(\forall g \in G \colon \, gVg^{-1} = V\right) .
\end{displaymath} For instance, if a metric $d$ on a group $G$ is \emph{bi-invariant}, in the sense that $d(gx,gy) = d(x,y) = d(xg,yg)$ for all $g,x,y \in G$, then $G$ endowed with the topology generated by $d$ constitutes a topological group with SIN.

\begin{lem}\label{lemma:continuity} Let $G$ be a topological group with SIN. If $\phi \in \PD(G)$, then \begin{displaymath}
    \cconv \!\left(\phi^{G}\right) \, \subseteq \, \PD(G) .
\end{displaymath} \end{lem}

\begin{proof} Let $\phi \in \PD(G)$. One readily checks that \begin{displaymath}
    \left. A \, \defeq \, \left\{ \psi \in \C^{G} \, \right\vert \psi \text{ positive-definite, } \psi(1) = 1 \right\}
\end{displaymath} is a $G$-invariant closed convex subset of $\C^{G}$, whence $\cconv \!\left(\phi^{G}\right) \subseteq A$. It thus remains to verify that every element of $\cconv \!\left(\phi^{G}\right)$ is continuous. Completing the proof, we will establish the following even stronger claim: \begin{align*}
    &\forall \epsilon \in \R_{>0} \ \exists U \in \Neigh(G) \ \forall \psi \in \cconv \!\left(\phi^{G}\right) \ \forall x,y \in G \colon \\
    &\qquad \qquad x^{-1}y \in U \ \Longrightarrow \ \vert \psi(x) - \psi(y) \vert \leq \epsilon .
\end{align*} To this end, let $\epsilon \in \R_{>0}$. Since $\phi$ is continuous and $G$ has SIN, there exists $U \in \Neigh(G)$ such that $\diam \phi(U) \leq \tfrac{\epsilon^{2}}{18}$ and $gUg^{-1} = U$ for all $g \in G$. Now, if $\psi \in \cconv \!\left(\phi^{G}\right)$ and $x,y \in G$ are such that $x^{-1}y \in U$, then there exist a finite set $F \subseteq G$ and $\mu \in [0,1]^{F}$ such that $\sum_{g \in F} \mu(g) = 1$ and $\left\lvert \psi(z) - \sum\nolimits_{g \in F} \mu(g)\phi^{g}(z) \right\rvert \leq \tfrac{\epsilon}{3}$ for each $z \in \{ x,y \}$, which entails that \begin{align*}
    \vert \psi(x) - \psi(y) \vert \, &\leq \, \left\lvert \psi(x) - \sum\nolimits_{g \in F} \mu(g)\phi^{g}(x) \right\rvert \\
    & \qquad \qquad + \left\lvert \sum\nolimits_{g \in F} \mu(g)\phi^{g}(x) - \sum\nolimits_{g \in F} \mu(g)\phi^{g}(y) \right\rvert \\
    & \qquad \qquad + \left\lvert \sum\nolimits_{g \in F} \mu(g)\phi^{g}(y) - \psi(y) \right\rvert \\
    &\leq \, \tfrac{2\epsilon}{3} + \sum\nolimits_{g \in F} \mu(g)\!\left\lvert \phi\!\left(gxg^{-1}\right) - \phi\!\left( gyg^{-1}\right) \right\rvert \\
    &\stackrel{\ref{lemma:positive.definite.functions}\ref{lemma:positive.definite.functions.2}}{\leq} \, \tfrac{2\epsilon}{3} + \sum\nolimits_{g \in F} \mu(g)\sqrt{2\!\left( 1- \Re \phi\!\left( gx^{-1}yg^{-1}\right)\right)} \, \leq \, \epsilon . \qedhere
\end{align*} \end{proof}

If $G$ is a topological group, then we let \begin{displaymath}
    \ch(G) \, \defeq \, \PD(G) \cap \Delta(G) \, = \, \{ \chi \in \PD(G) \mid \forall g \in G \colon \, \chi^{g} = \chi \} .
\end{displaymath}

\begin{lem}\label{lemma:0} Let $G$ be a topological group with SIN and let $\phi \in \PD(G)$. Suppose that $G$ admits a set $\mathcal{H}$ of subgroups such that \begin{enumerate}
    \item[---\,] $(\mathcal{H},{\subseteq})$ is directed,
    \item[---\,] $G = \overline{\bigcup \mathcal{H}}$, and
    \item[---\,] $\cconv \!\left((\phi\vert_{H})^{H}\right) \cap \Delta(H) \ne \emptyset$ for each $H \in \mathcal{H}$.
\end{enumerate} Then $\cconv \!\left(\phi^{G}\right) \cap \ch(G) \ne \emptyset$. \end{lem}

\begin{proof} Let $\phi \in \PD(G)$. Then $\Vert \phi \Vert_{\infty} \leq 1$ by Lemma~\ref{lemma:positive.definite.functions}\ref{lemma:positive.definite.functions.1} and thus \begin{displaymath}
    {\cconv \!\left(\phi^{G}\right)} \, \subseteq \, \{ z \in \C \mid \vert z \vert \leq 1 \}^{G} ,
\end{displaymath} whence Tychonoff's theorem implies that ${\cconv \!\left(\phi^{G}\right)}$ is compact, which immediately entails compactness of the closed subset ${\cconv \!\left(\phi^{H}\right)} \subseteq {\cconv \!\left(\phi^{G}\right)}$ for every subgroup $H \leq G$. For any subgroup $H \leq G$, since \begin{displaymath}
    \Delta(H) \, = \, \bigcap\nolimits_{g,h \in H} \left. \! \left\{ \psi \in \C^{H} \, \right\vert \psi\!\left(ghg^{-1}\right) = \psi(h) \right\}
\end{displaymath} is closed in $\C^{H}$ and the map \begin{displaymath}
    T_{H} \colon \, \C^{G} \, \longrightarrow \, \C^{H}, \quad \psi \, \longmapsto \, \psi\vert_{H}
\end{displaymath} is continuous, we see that \begin{displaymath}
    D(H) \, \defeq \, T_{H}^{-1}(\Delta(H))
\end{displaymath} is closed in $\C^{G}$. Furthermore, as $(\mathcal{H},{\subseteq})$ is directed and \begin{displaymath}
    \forall H_{0},H_{1} \in \mathcal{H} \colon \qquad H_{0} \subseteq H_{1} \, \Longrightarrow \, D(H_{0}) \supseteq D(H_{1}) ,
\end{displaymath} we know that $(\{ D(H) \mid H \in \mathcal{H}\},{\supseteq})$ is directed, too. We claim that \begin{equation}\label{eq:nonempty}
    \forall H \in \mathcal{H} \colon \qquad {\cconv \!\left(\phi^{G}\right)} \cap {D(H)} \ne \emptyset .
\end{equation} To see this, consider any $H \in \mathcal{H}$. Since $T_{H}$ is continuous, linear, and $H$-equivariant, and ${\cconv \!\left(\phi^{H}\right)}$ is compact, we infer that \begin{displaymath}
    T_{H}\!\left( \cconv \!\left(\phi^{H}\right) \right) \, = \, \overline{T_{H}\!\left(\conv \!\left(\phi^{H}\right)\right)} \, = \, \cconv \!\left(T_{H}(\phi)^{H}\right) \, = \, \cconv \!\left((\phi\vert_{H})^{H}\right) .
\end{displaymath} In turn, \begin{displaymath}
    T_{H}\!\left( \cconv \!\left(\phi^{H}\right) \right) \cap \Delta(H) \, = \, \cconv \!\left((\phi\vert_{H})^{H}\right) \cap \Delta(H) \, \ne \, \emptyset
\end{displaymath} by assumption, wherefore \begin{displaymath}
    {\cconv \!\left(\phi^{H}\right)} \cap {D(H)} \, = \, {\cconv \!\left(\phi^{H}\right)} \cap {T_{H}^{-1}(\Delta(H))} \, \ne \, \emptyset
\end{displaymath} and so ${\cconv \!\left(\phi^{G}\right)} \cap {D(H)} \ne \emptyset$, which proves~\eqref{eq:nonempty}. From~\eqref{eq:nonempty} and compactness of $\cconv \!\left(\phi^{G}\right)$, we now deduce that \begin{displaymath}
    \cconv \!\left(\phi^{G}\right) \cap \bigcap\nolimits_{H \in \mathcal{H}} D(H) \, \ne \, \emptyset ,
\end{displaymath} that is, there exists $\chi \in \cconv \!\left(\phi^{G}\right) \cap \bigcap\nolimits_{H \in \mathcal{H}} D(H)$. Then $\chi \in \PD(G)$ by Lemma~\ref{lemma:continuity}. In particular, $\chi$ is continuous and so \begin{displaymath}
    C \, \defeq \, \left\{ (g,h) \in G \times G \left\vert \, \chi\!\left(ghg^{-1}\right) = \chi(h) \right\} \right.
\end{displaymath} is closed in $G \times G$. As $(\bigcup \mathcal{H}) \! \times \! (\bigcup \mathcal{H})$ is dense in $G \times G$ and \begin{displaymath}
    \left(\bigcup \mathcal{H}\right) \! \times \! \left(\bigcup \mathcal{H}\right) \, = \, \bigcup\nolimits_{H \in \mathcal{H}} H \times H \, \subseteq \, C,
\end{displaymath} it follows that $C = G \times G$, i.e., $\chi \in \Delta(G)$. Hence, $\chi \in \ch(G)$ as desired. \end{proof}

\section{Continuous geometries and their coordinate rings}\label{section:continuous.rings}

\begin{flushright}
	\begin{minipage}[t]{0.85\linewidth}\itshape
		John von Neumann's brilliant mind blazed over lattice theory like a meteor, during a brief period centering around 1935--1937.
	\end{minipage}
\end{flushright}
\begin{flushright}
	Garrett Birkhoff, \cite[p.~50]{birkhoff}
\end{flushright}

A \emph{continuous geometry}~\cite{VonNeumannContinuousGeometry} is a complete, complemented, modular lattice~$L$ such that, for every chain $C \subseteq L$ and every $x \in L$, \begin{displaymath}
    x \wedge \bigvee C = \bigvee \{ x \wedge y \mid y \in C \}, \qquad x \vee \bigwedge C = \bigwedge \{ x \vee y \mid y \in C \} .
\end{displaymath} We give an overview of these objects, their coordinate rings and the associated unit groups. For more details, the reader is referred to von Neumann's original work~\cite{VonNeumannBook} as well as the monographs by Maeda~\cite{MaedaBook} and Goodearl~\cite{GoodearlBook}.

We start off with the coordinatization by regular rings. A unital ring $R$ is called \emph{(von Neumann) regular} if, for every $a \in R$, there exists $b \in R$ such that $aba = a$. Due to~\cite[Theorem~2]{VonNeumannRegularRings} (see also~\cite[II.II, Theorem~2.4, p.~72]{VonNeumannBook}), if $R$ is a regular ring, then the set \begin{displaymath}
    \lat(R) \, \defeq \, \{ aR \mid a \in R \} ,
\end{displaymath} partially ordered by inclusion, constitutes a complemented, modular lattice. Conversely,  von Neumann's coordinatization theorem~\cite[II.XIV, Theorem~14.1, p.~208]{VonNeumannBook}\footnote{This was announced in~\cite{VonNeumannAlgebraicTheory}.} asserts that, for any complemented, modular lattice $L$ of an order at least four, there exists a regular ring $R$, unique up to isomorphism, such that $L \cong \lat(R)$.

Recall that an algebraic structure is said to be \emph{(directly) irreducible} if it contains at least two elements and it is not isomorphic to a direct product of two structures (of the same type) of cardinality at least two. Every continuous geometry is a subdirect product of irreducible continuous geometries~\cite[V.3, Satz~3.2, p.~128]{MaedaBook}. The property of irreducibility translates conveniently via coordinatization and, moreover, is reflected in the center of the corresponding ring (Remark~\ref{remark:irreducible}). The \emph{center} of a unital ring $R$ will be denoted by \begin{displaymath}
    \cent(R) \, \defeq \, \{ a \in R \mid \forall b \in R \colon \, ab = ba \} .
\end{displaymath}

\begin{remark}[\cite{VonNeumannBook}]\label{remark:irreducible} Let $R$ be a regular ring. The following are equivalent. \begin{enumerate}
	\item\label{remark:irreducible.1} $R$ is irreducible.
	\item\label{remark:irreducible.2} $\cent(R)$ is a field.
	\item\label{remark:irreducible.3} $\lat(R)$ is irreducible.
\end{enumerate} The equivalence \ref{remark:irreducible.1}$\Longleftrightarrow$\ref{remark:irreducible.2} is due to~\cite[II.II, Theorem~2.7, p.~75]{VonNeumannBook}, while \ref{remark:irreducible.2}$\Longleftrightarrow$\ref{remark:irreducible.3} is established in~\cite[II.II, Theorem~2.9, p.~76]{VonNeumannBook}. \end{remark}

The \emph{unit group} of a unital ring $R$, i.e., the group of multiplicatively invertible elements of $R$, will be denoted by \begin{displaymath} 
    \GL(R) \, \defeq \, \{ a \in R \mid \exists b \in R \colon \, ab = ba = 1 \} .
\end{displaymath} Preparing the statement of the subsequent proposition, let us also recall that the \emph{center} of a group $G$ is defined as \begin{displaymath}
    \cent(G) \, \defeq \, \{ g \in G \mid \forall h \in G \colon \, gh = hg \} \, \unlhd \, G .
\end{displaymath}

\begin{prop}[{\cite[Proposition~3.11]{Schneider25a}}, see also~{\cite[Lemma~8]{Ehrlich56}}]\label{proposition:center} Let $R$ be an irreducible, continuous ring. Then $\cent(\GL(R)) = \cent(R) \setminus \{ 0 \}$. \end{prop}

We proceed to von Neumann's dimension theory. According to~\cite[I.VI, Theorem~6.9, p.~52]{VonNeumannBook} and~\cite[I.VII, p.~60, Corollary~1]{VonNeumannBook}, every irreducible continuous geometry $L$ admits a unique map $\delta_{L} \colon L \to [0,1]$, called the \emph{dimension function} of $L$, such that \begin{itemize}
	\item[---\,] $\delta_{L}(0) = 0$ and $\delta_{L}(1) = 1$,
	\item[---\,] $\delta_{L}(x) < \delta_{L}(y)$ for all $x,y \in L$ with $x < y$, and
	\item[---\,] $\delta_{L}(x \vee y) + \delta_{L}(x \wedge y) = \delta_{L}(x) + \delta_{L}(y)$ for all $x,y \in L$.
\end{itemize} This dimension function gives rise to a rank function on the corresponding coordinate ring (Theorem~\ref{theorem:rank.function}). Given a ring $R$, let $\E(R) \defeq \{ e \in R \mid ee = e \}$. A \emph{rank function} on a regular ring $R$ is a map $\rho \colon R \to [0,1]$ such that \begin{itemize}
	\item[---\,] $\rho(1) = 1$,
	\item[---\,] $\rho(ab) \leq \min \{ \rho(a), \rho(b) \}$ for all $a,b \in R$,
	\item[---\,] $\rho(e+f) = \rho(e) + \rho(f)$ for all $e,f \in \E(R)$ with $ef = fe = 0$,\footnote{This readily implies that $\rho(0) = 0$.} and
	\item[---\,] $\rho(a) > 0$ for every $a \in R\setminus \{ 0 \}$.
\end{itemize} Such a rank function naturally generates a metric and, in turn, a topology on the underlying ring as follows.

\begin{remark}[\cite{VonNeumannBook}, see~{\cite[Remark~7.8]{SchneiderGAFA}} for detailed references]\label{remark:rank.topology} Let $\rho$ be a rank function on a regular ring $R$. Then \begin{displaymath}
	d_{\rho} \colon \, R \times R \, \longrightarrow \, [0,1], \quad (a,b) \, \longmapsto \, \rho(a-b)
\end{displaymath} is a metric on $R$. The \emph{$\rho$-topology}, i.e., the topology generated by $d_{\rho}$, turns $R$ into a Hausdorff topological ring. Since $d_{\rho}$ restricts to a bi-invariant metric on $\GL(R)$, the latter endowed with the relative $\rho$-topology constitutes a topological group with SIN. \end{remark}

\begin{thm}[\cite{VonNeumannBook}]\label{theorem:rank.function} Let $R$ be an irreducible, continuous ring. Then \begin{displaymath}
	\rk_{R} \colon \, R \, \longrightarrow \, [0,1], \quad a \, \longmapsto \, \delta_{\lat(R)}(aR)
\end{displaymath} is the unique rank function on $R$. Moreover, the metric $d_{R} \defeq d_{\rk_{R}}$ is complete. \end{thm}

\begin{proof} The map $\rk_{R}$ is a rank function on $R$ by~\cite[II.XVII, Theorem~17.1, p.~224]{VonNeumannBook} and unique as such due to~\cite[II.XVII, Theorem~17.2, p.~226]{VonNeumannBook}. Completeness of $(R,d_{R})$ is proved in~\cite[II.XVII, Theorem~17.4, p.~230]{VonNeumannBook}. \end{proof}

\begin{remark}[cf.~{\cite[Remark~4.5(B)]{BernardSchneider}}]\label{remark:rank.function} Let $R$ be an irreducible, continuous ring. If $\rho$ is a rank function on a regular ring $S$ and $\phi \colon R \to S$ is a unital ring homomorphism, then $\rho \circ \phi = \rk_{R}$. \end{remark}

An irreducible, continuous ring $R$ is called \emph{discrete} if its \emph{rank topology}, i.e., the $\rk_{R}$-topology in the sense of Remark~\ref{remark:rank.topology}, is discrete. By work of von Neumann~\cite{VonNeumannBook} (see~\cite[Remark~3.6]{SchneiderIMRN} for detailed references), the discrete irreducible, continuous rings are precisely the ones isomorphic to a matrix ring $\M_{n}(D)$ for some division ring $D$ and some $n \in \N_{>0}$. Of course, if $D$ is a division ring and $n \in \N_{>0}$, then \begin{displaymath}
    \rk_{\M_{n}(D)}(a) \, = \, \tfrac{\rank(a)}{n} \qquad (a \in \M_{n}(D)) 
\end{displaymath} due to~\cite[I.10.12, p.~359--360]{bourbaki} and the uniqueness asserted by Theorem~\ref{theorem:rank.function}. The following remark describes a basic example of a non-discrete irreducible, continuous ring.

\begin{remark}\label{remark:simple.example} Let $F$ be a field. Consider the inductive limit $\lim_{i \to \infty} M_{2^{i}}(F)$ of the unital $F$-algebra embeddings \begin{displaymath}
    \iota_{i} \colon \, \M_{2^{i}}(F) \, \longrightarrow \, \M_{2^{i+1}}(F), \quad a \, \longmapsto \, \begin{pmatrix} a & 0 \\ 0 & a \end{pmatrix} \qquad (i \in \N) .
\end{displaymath} One readily checks that ${\rk_{\M_{2^{i+1}}(F)}} \circ {\iota_{i}} = {\rk_{\M_{2^{i}}(F)}}$ for each $i \in \N$. Hence, the considered rank functions jointly extend to a unique rank function on $\lim_{i \to \infty} M_{2^{i}}(F)$, and it turns out that the completion of $\lim_{i \to \infty} M_{2^{i}}(F)$ with respect to the corresponding metric constitutes a non-discrete irreducible, continuous ring~\cite{VonNeumannExamples,Halperin68}, which we denote by\footnote{By~\cite[Theorem~1]{Halperin68} (see also~\cite[Theorem~2.2]{AraClaramunt}), replacing $(2^{i})_{i \in \N}$ in this construction by any other \emph{factor sequence} in the sense of~\cite[\S2]{Halperin68} results in a ring isomorphic to $\M_{\infty}(F)$.} $\M_{\infty}(F)$. By~\cite[Theorem~2.8(c)]{Goodearl78}, the natural embedding $F \hookrightarrow \cent(\M_{\infty}(F))$ is an isomorphism. We define \begin{displaymath}
    \A(F) \, \defeq \, \GL(\M_{\infty}(F)) .
\end{displaymath} It was shown in~\cite[p.~258]{CarderiThom} (see also~\cite[Lemma~6.14]{Schneider25a}) that \begin{displaymath}
    \A(F) \, = \, \overline{\lim\nolimits_{i \to \infty} \SL_{2^{i}}(F)} ,
\end{displaymath} with the inductive limit constructed from the restrictions of the maps $(\iota_{i})_{i \in \N}$. \end{remark}

Algebras of the kind constructed in Remark~\ref{remark:simple.example} are contained in any non-discrete irreducible, continuous ring (Lemma~\ref{lemma:embedding}).

\begin{lem}[{\cite[Lemma~9.3]{BernardSchneider}}]\label{lemma:matrix.subalgebras} Let $R$ be a non-discrete irreducible, continuous ring, let $F \defeq \cent(R)$ and $m,n \in \N_{>0}$. Every unital $F$-subalgebra of $R$ isomorphic to $\M_{n}(F)$ is contained in some unital $F$-subalgebra of $R$ isomorphic to $\M_{mn}(F)$. \end{lem}

\begin{lem}\label{lemma:skolem.noether} Let $F$ be a field, let $m,n \in \N_{>0}$, let $\iota \colon \M_{n}(F) \to \M_{mn}(F)$ be a unital $F$-algebra homomorphism, and let $R$ be a unital $F$-algebra isomorphic to $\M_{mn}(F)$. If $\phi \colon \M_{n}(F) \to R$ is a unital $F$-algebra homomorphism, then there exists an $F$-algebra isomorphism $\psi \colon \M_{mn}(F) \to R$ such that $\psi \circ \iota = \phi$. \end{lem}

\begin{proof} Let $\phi \colon \M_{n}(F) \to R$ be a unital $F$-algebra homomorphism. Fix any $F$-algebra isomorphism $\theta \colon \M_{mn}(F) \to R$. Thanks to the Skolem--Noether theorem (see, e.g.,~\cite[I.3, Theorem~3.14, p.~93]{FarbDennis}), there exists $b \in \GL(R)$ with $b\theta(\iota(a))b^{-1} = \phi(a)$ for all $a \in \M_{n}(F)$. So, $\psi \colon \M_{mn}(F) \to R, \, a \mapsto b\theta(a)b^{-1}$ constitutes an $F$-algebra isomorphism such that $\psi \circ \iota = \phi$. \end{proof}

\begin{lem}\label{lemma:embedding} Let $R$ be a non-discrete irreducible, continuous ring, $F \defeq \cent(R)$. Then there exists a unital $F$-algebra embedding from $\M_{\infty}(F)$ into $R$. \end{lem}

\begin{proof} Consider the inductive limit $S \defeq \lim_{i \to \infty} \M_{2^{i}}(F)$ of the sequence of unital $F$-algebra embeddings \begin{displaymath}
    \iota_{i} \colon \, \M_{2^{i}}(F) \, \longrightarrow \, \M_{2^{i+1}}(F), \quad a \, \longmapsto \, \begin{pmatrix} a & 0 \\ 0 & a \end{pmatrix} \qquad (i \in \N) .
\end{displaymath} An iterated application of Lemma~\ref{lemma:matrix.subalgebras} yields the existence of an ascending chain of unital $F$-subalgebras $(T_{i})_{i \in \N}$ of $R$ such that $T_{i} \cong \M_{2^{i}}(F)$ for each $i \in \N$. Using Lemma~\ref{lemma:skolem.noether} inductively, we find a sequence of $F$-algebra isomorphisms \begin{displaymath}
    \phi_{i} \colon \, \M_{2^{i}}(F) \, \longrightarrow \, T_{i} \qquad (i \in \N)
\end{displaymath} such that ${\phi_{i+1}} \circ {\iota_{i}} = {\phi_{i}}$ for each $i \in \N$. Therefore, $(\phi_{i})_{i \in \N}$ jointly extend to a unique unital $F$-algebra embedding $\phi \colon S \to R$. Moreover, Remark~\ref{remark:rank.function} asserts that ${\rk_{R}} \circ {\phi_{i}} = {\rk_{\M_{2^{i}}(F)}}$ for all $i \in \N$, whence ${{\rk_{\M_{\infty}(F)}}\vert_{S}} = {\rk_{R}} \circ {\phi}$. Consequently, \begin{displaymath}
    \phi \colon \, \!\left(S,{d_{\M_{\infty}(F)}\vert_{S \times S}}\right)\! \, \longrightarrow \, (R,{d_{R}})
\end{displaymath} is isometric. Since $(R,d_{R})$ is complete by Theorem~\ref{theorem:rank.function}, the map $\phi$ extends to a unique isometric map \begin{displaymath}
    \overline{\phi} \colon \, \! \left(\M_{\infty}(F),{d_{\M_{\infty}(F)}}\right)\! \, \longrightarrow \, (R,d_{R}) .
\end{displaymath} As $\M_{\infty}(F)$ and $R$ are Hausdorff topological rings by Remark~\ref{remark:rank.topology}, thus $\overline{\phi}$ is a ring homomorphism and, in turn, a unital $F$-algebra embedding. \end{proof}

\section{Charmenability and lack of characters}\label{section:charmenability}

In this section, we study continuous positive-definite functions and characters on the topological group $\A(\Q)$ (see Remark~\ref{remark:simple.example}). We start off with an application of a character rigidity result due to Kirillov~\cite{Kirillov} (see also~\cite{PetersonThom}).

\begin{lem}\label{lemma:3} If $F$ is an infinite field, then $\ch(\A(F)) = \{ 1 \}$. \end{lem}

\begin{proof} Let $\chi \in \ch(\A(F))$. Since $\A(F) = \overline{\lim_{n \to \infty} \SL_{2^{n}}(F)}$ by Remark~\ref{remark:simple.example}, it suffices to check that \begin{displaymath}
	\forall \epsilon \in \R_{>0} \, \forall m \in \N \, \exists n \in \N_{>m} \colon \quad \Vert 1- {\chi\vert_{\SL_{2^{n}}(F)}} \Vert_{\infty} \leq \epsilon .
\end{displaymath} So, let $\epsilon \in \R_{>0}$ and $m \in \N$. As $\chi$ is continuous, there exists $\delta \in (0,1]$ such that $\Re \chi(g) \geq 1-\tfrac{\epsilon}{2}$ for all $g \in \A(F)$ with $\rk_{\M_{\infty}(F)}(1-g) \leq \delta$. Choose $n \in \N_{>m}$ such that $\tfrac{2}{2^{n}} \leq \delta$. Being a convex compact subset of the Hausdorff locally convex topological vector space $\C^{\SL_{2^{n}}(F)}$, the set $\ch(\SL_{2^{n}}(F))$ coincides with the topological closure of the convex hull of the set of its extreme points, due to the Krein--Milman theorem. By~\cite[Theorem~1]{Kirillov} (see also~\cite[Theorem~2.4]{PetersonThom}), the only extreme points of $\ch(\SL_{2^{n}}(F))$ are the trivial character $1$ and the ones induced from the center $C \defeq \cent(\SL_{2^{n}}(F)) \cong \left\{ a \in F \left\vert \, a^{2^{n}} = 1 \right\} \!\right.$, i.e., belonging to the finite set \begin{displaymath}
	\left. T \, \defeq \, \left\{ \tau \in \C^{\SL_{2^{n}}(F)} \, \right\vert \tau\vert_{C} \in \Hom(C,\T), \, \tau\vert_{\SL_{2^{n}}(F)\setminus C} = 0 \right\} .
\end{displaymath} It follows that \begin{align*}
	\ch(\SL_{2^{n}}(F)) \, &= \, \cconv (\{ 1 \} \cup  T) \, = \, \conv (\{ 1 \} \cup T) \\
    &= \, \{ \lambda + (1-\lambda) \tau \mid \lambda \in [0,1], \, \tau \in \conv (T) \} .
\end{align*} In particular, there exist $\lambda \in [0,1]$ and $\tau \in \conv (T)$ with $\chi\vert_{\SL_{2^{n}}(F)} = \lambda + (1-\lambda)\tau$. Of course, there exists some $g \in \SL_{2^{n}}(F) \setminus C$ such that \begin{displaymath}
    \rk_{\M_{2^{n}}(F)}(1-g) \, \leq \, \tfrac{2}{2^{n}} \, \leq \, \delta .
\end{displaymath} Since $\tau\vert_{\SL_{2^{n}}(F)\setminus C} = 0$, we conclude that \begin{displaymath}
    \lambda \, = \, \lambda + (1-\lambda) \Re \tau(g) \, = \, \Re \chi(g) \, \geq \, 1-\tfrac{\epsilon}{2} .
\end{displaymath} Consequently, \begin{align*}
	\Vert 1- {\chi\vert_{\SL_{2^{n}}(F)}} \Vert_{\infty} \, &\leq \, \Vert 1 - \lambda \Vert_{\infty} + \Vert \lambda - {\chi\vert_{\SL_{2^{n}}(F)}} \Vert_{\infty} \\
	& \leq \, 1-\lambda + (1-\lambda)\Vert \tau \Vert_{\infty} \, \leq \, 2-2\lambda \, \leq \, \epsilon .\qedhere
\end{align*} \end{proof}

We now turn to applications of the work of Bader--Boutonnet--Houdayer--Peterson~\cite{BaderBoutonnetHoudayerPeterson} on \emph{charmenability} of certain special linear groups.

\begin{lem}\label{lemma:1} Let $n \in \N$ and consider the discrete group $G \defeq \SL_{n}(\Q)$. If $\phi \in \PD(G)$, then $\cconv \!\left(\phi^{G}\right) \cap \ch(G) \ne \emptyset$. \end{lem}

\begin{proof} Consider the set $\mathcal{H} \defeq \{\SL_{n}(\Z[1/m]) \mid m \in \N_{>1} \}$ of subgroups of $G$. As \begin{displaymath}
    \SL_{n}(\Z[1/m]) \cup \SL_{n}(\Z[1/m']) \, \subseteq \, \SL_{n}(\Z[1/mm']) 
\end{displaymath} for any two $m,m' \in \N_{>1}$, we see that $(\mathcal{H},{\subseteq})$ is directed. Evidently, $G = \bigcup \mathcal{H}$. Now, if $\phi \in \PD(G)$, then~\cite[Theorem~A]{BaderBoutonnetHoudayerPeterson} asserts that $\cconv \!\left((\phi\vert_{H})^{H}\right) \cap \Delta(H) \ne \emptyset$ for every $H \in \mathcal{H}$, whence $\cconv \!\left(\phi^{G}\right) \cap \ch(G) \ne \emptyset$ thanks to Lemma~\ref{lemma:0}. \end{proof}

\begin{lem}\label{lemma:2} If $\phi \in \PD(\A(\Q))$, then $1 \in \cconv \!\left(\phi^{\A(\Q)}\right)$. \end{lem}

\begin{proof} Appealing to Remark~\ref{remark:simple.example}, we identify $\mathcal{H} \defeq \{ \SL_{2^{i}}(\Q) \mid i \in \N \}$ as a chain of subgroups of $G \defeq \A(\Q)$ such that $G = \overline{\bigcup \mathcal{H}}$. Now, if $\phi \in \PD(G)$, then Lemma~\ref{lemma:1} yields that $\cconv \!\left((\phi\vert_{H})^{H}\right) \cap \Delta(H) \ne \emptyset$ for every $H \in \mathcal{H}$, wherefore $\cconv \!\left(\phi^{G}\right) \cap \ch(G) \ne \emptyset$ by Lemma~\ref{lemma:0}. As $\ch(\A(\Q)) = \{ 1 \}$ by Lemma~\ref{lemma:3}, this entails that $1 \in \cconv \!\left(\phi^{\A(\Q)}\right)$. \end{proof}

\begin{lem}\label{lemma:kernel} If $\pi$ is a continuous unitary representation of $\A(\Q)$, then \begin{displaymath}
    \cent(\A(\Q)) \, \subseteq \, \Ker \pi .
\end{displaymath} \end{lem}

\begin{proof} According to Lemma~\ref{lemma:exoticness.via.positive.definite.functions}, it suffices to check that \begin{displaymath}
    \forall \phi \in \PD(A(\Q)) \ \forall h \in \cent(\A(\Q)) \ \forall \epsilon \in \R_{>0} \colon \quad \vert 1 - \phi(h) \vert \leq \epsilon .
\end{displaymath} To this end, let $\phi \in \PD(A(\Q))$. Then $1 \in \cconv \!\left(\phi^{\A(\Q)}\right)$ thanks to Lemma~\ref{lemma:2}. Hence, if $h \in \cent(\A(\Q))$ and $\epsilon \in \R_{>0}$, then there exist a finite set~$F \subseteq \A(\Q)$ and $\mu \in [0,1]^{F}$ such that $\sum_{g \in F} \mu(g) = 1$ and $\left\lvert 1 - \sum\nolimits_{g \in F} \mu(g)\phi^{g}(h) \right\rvert \leq \epsilon$, wherefore \begin{displaymath}
    \vert 1 - \phi(h) \vert \, \stackrel{h \in \cent(\A(\Q))}{=} \, \left\lvert 1 - \sum\nolimits_{g \in F} \mu(g)\phi^{g}(h) \right\rvert \, \leq \, \epsilon .\qedhere
\end{displaymath} \end{proof}

\section{Proof of the main result}\label{section:main.result}

This section contains the proofs of Theorem~\ref{theorem:exotic} and Corollaries~\ref{corollary:whirly} and~\ref{corollary:separable.reps}.

\begin{proof}[Proof of Theorem~\ref{theorem:exotic}] Let $\pi$ be a continuous unitary representation of $\GL(R)$. Let $F$ denote the prime field of $\cent(R)$. By case analysis, we show that \begin{displaymath}
    \Ker \pi \, \nsubseteq \, \cent(R) .
\end{displaymath}

\emph{Case 1}: $F$ is finite. By Lemma~\ref{lemma:embedding} and the fact that $\M_{\infty}(F)$ naturally embeds into $\M_{\infty}(\cent(R))$, there exists a unital ring embedding $\iota \colon \M_{\infty}(F) \to R$, which necessarily satisfies ${\rk_{R}} \circ \iota = {\rk_{\M_{\infty}(F)}}$ by Remark~\ref{remark:rank.function} and thus restricts to an embedding of the topological group $\A(F)$ into $\GL(R)$. As $\A(F)$ is exotic due to~\cite{CarderiThom}, the continuous unitary representation $\pi \circ {\iota\vert_{\A(F)}}$ is trivial, i.e., $\iota(\A(F)) \subseteq \Ker \pi$. Containing a subgroup isomorphic to $\SL_{2}(F)$, the group $\iota(\A(F))$ is non-abelian and hence not contained in $\cent(R)$, which entails that $\Ker \pi \nsubseteq \cent(R)$.

\emph{Case 2}: $F$ is infinite, i.e., $F \cong \Q$. As $R$ is not discrete, in particular there exist $z,z' \in R$ with $d_{R}(z,z') \notin \{ 0,1 \}$. Let $x \defeq z-z'$. By regularity of $R$, there exists $y \in R$ such that $xyx = x$. Straightforward calculations show that $e \defeq xy \in \E(R)$ and $\rk_{R}(e) = \rk_{R}(x) = \rk_{R}(z-z') = d_{R}(z,z') \notin \{ 0,1 \}$. Consequently, $e \in \E(R) \setminus \{ 0,1 \}$. Since $e \ne 0$, we know that $S \defeq eRe$ is a non-discrete irreducible, continuous ring with $\cent(S) = \cent(R)e \cong \cent(R)$ (see, e.g.,~\cite[Remark~4.8]{BernardSchneider}). Hence, due to Lemma~\ref{lemma:embedding} and the fact that $\M_{\infty}(F)$ naturally embeds into $\M_{\infty}(\cent(R))$, there exists a unital $F$-algebra embedding $\iota \colon \M_{\infty}(F) \to S$. Moreover, ${\rk_{S}} \circ \iota = {\rk_{\M_{\infty}(F)}}$ thanks to Remark~\ref{remark:rank.function}. In turn, the restriction ${\iota\vert_{\A(F)}} \colon \A(F) \to \GL(S)$ constitutes an embedding of topological groups. Furthermore, by~\cite[Lemma~6.1]{Schneider25a}, \begin{displaymath}
    \phi \colon \, \GL(S) \, \longrightarrow \, \GL(R), \quad a \, \longmapsto \, a + 1 - e
\end{displaymath} is an embedding of topological groups, too. Thus, ${\pi'} \defeq \pi \circ \phi \circ {\iota\vert_{\A(F)}}$ is a continuous unitary representation of $\A(F)$. Therefore, \begin{displaymath}
    F\setminus \{ 0 \} \, \stackrel{\ref{remark:simple.example}}{\cong}  \, \cent(\M_{\infty}(F)) \setminus \{ 0 \} \, \stackrel{\ref{proposition:center}}{=} \, \cent(\A(F)) \, \stackrel{\ref{lemma:kernel}}{\subseteq} \, \Ker {\pi'} ,
\end{displaymath} and so \begin{displaymath}
    T \, \defeq \, \{ ae + 1 - e \mid a \in F \setminus \{ 0 \} \} \, = \, \phi(\iota(\cent(\A(F)))) \, \subseteq \, \Ker \pi .
\end{displaymath} As $R$ is an irreducible unital ring, we know that \begin{displaymath}
    \E(R) \cap \cent(R) \, = \, \{ 0,1\}
\end{displaymath} (see, e.g.,~\cite[VI.1, Satz~1.10, p.~139]{MaedaBook}). Due to~\cite[II.II, Lemma~2.8, p.~76]{VonNeumannBook}, \begin{displaymath}
    \E(R) \cap \cent(R) \, = \, \{ f \in \E(R) \mid fR(1-f) = \{ 0 \} \} .
\end{displaymath} Since $e \notin \{ 0,1 \}$, we conclude that $eR(1-e) \ne \{ 0 \}$. Fix any $b \in eR(1-e) \setminus \{ 0 \}$. Now, if $a \in F \setminus \{ 0 \}$ is such that $tb = bt$ for $t \defeq ae+1-e \in T$, then \begin{displaymath}
    ab \, = \, aeb+(1-e)b \, = \, tb \, = \, bt \, = \, bae+b(1-e) \, = \, abe+b(1-e) \, = \, b
\end{displaymath} and thus $(1-a)b = 0$, which due to $b \ne 0$ and $F$ being a field implies that $1-a = 0$, i.e., $a=1$. As $F \setminus \{ 0,1 \} \ne \emptyset$, hence $T \nsubseteq \cent(R)$, so $\Ker \pi \nsubseteq \cent(R)$.

By~\cite[Theorem~A]{Schneider25a}, since $\Ker \pi$ is a closed normal subgroup of $\GL(R)$ not contained in $\cent(R)$, it follows that $\Ker \pi = \GL(R)$, i.e., $\pi$ is trivial. \end{proof}

\begin{remark} While the proof of Theorem~\ref{theorem:exotic} given above uses the exoticness result of~\cite{CarderiThom}, the former actually suggests a simplified proof of the latter. In order to sketch this proof, let $F$ be a finite field. Then $\ch(\A(F)) = 1$ by the argument suggested in~\cite[Remark~1.3]{CarderiThom}. From this and Lemma~\ref{lemma:0}, applied to the set $\mathcal{H} \defeq \{ \SL_{2^{n}}(F) \mid n \in \N \}$ of finite subgroups of $\A(F)$, we infer that $1 \in \cconv \!\left(\phi^{\A(F)}\right)$ for every $\phi \in \PD(\A(F))$. Hence, as in Lemma~\ref{lemma:kernel}, the kernel of every continuous unitary representation of $\A(F)$ contains $\cent(\A(F))$. Arguing as in \emph{Case 2} of the proof of Theorem~\ref{theorem:exotic}, we conclude that $\A(F)$ is exotic if $\vert F \vert > 2$. To cover remaining case where $F \cong \F_{2}$, one may use a unital ring embedding $\F_{4} \hookrightarrow \M_{2}(\F_{2})$ to construct a unital ring embedding \begin{displaymath}
	\lim\nolimits_{i\to \infty} \M_{2^{i}}(\F_{4}) \, \lhook\joinrel\longrightarrow \, \lim\nolimits_{i\to \infty} \M_{2^{i}}(\M_{2}(\F_{2})) \, \cong \, \lim\nolimits_{i\to \infty} \M_{2^{i}}(\F_{2}) ,
\end{displaymath} which is necessarily compatible with the rank functions thanks to Remark~\ref{remark:rank.function} and thus extends to a unital ring embedding $\M_{\infty}(\F_{4}) \hookrightarrow \M_{\infty}(\F_{2})$. In turn, the topological group $\A(\F_{4})$ embeds into $\A(\F_{2})$. Since $\A(\F_{4})$ is non-abelian and exotic by the preceding argument, $\A(\F_{2})$ is exotic by~\cite[Proposition~3.2]{CarderiThom}. \end{remark}

Before turning to the announced corollaries, let us briefly clarify the connection between exoticness and whirlyness.

\begin{remark}[{\cite[Remark~5.7]{SchneiderSolecki}}]\label{remark:whirly} Let $G$ be a topological group acting continuously on a compact Hausdorff space $X$ and let $\mu$ be a $G$-invariant regular Borel probability measure on $X$. Then the support of $\mu$ is contained in the set of $G$-fixed points in $X$ if and only if the induced continuous homomorphism $\pi \colon G \to \U(L^{2}(\mu))$ is trivial. The forward implication is obvious. Conversely, if the support of $\mu$ is not contained in the set of $G$-fixed points, then we find an open set $U \subseteq X$ with $\mu(U) > 0$ and some $g \in G$ such that $U \cap gU = \emptyset$, whence $\Vert \chi_{U} - \pi(g)\chi_{U} \Vert^{2}_{2,\mu} = 2\mu(U) > 0$ and so $\pi \ne 1$. In particular, exoticness implies whirlyness. The converse does not hold: for example, the unitary group of an infinite-dimensional Hilbert space equipped with the strong operator topology is whirly~\cite[Theorem~1.1]{GlasnerTsirelsonWeiss}, but not exotic. \end{remark}

\begin{proof}[Proof of Corollary~\ref{corollary:whirly}] This follows from Theorem~\ref{theorem:exotic} and Remark~\ref{remark:whirly}. \end{proof}

\begin{proof}[Proof of Corollary~\ref{corollary:separable.reps}] Let $H$ be a separable Hilbert space. Since the strong operator topology constitutes a separable group topology on $\U(H)$ (see, e.g.,~\cite[I.9, Example~9.B(6), p.~59]{KechrisBook}), any homomorphism $\pi \colon \GL(R) \to \U(H)$ must be continuous by~\cite[Theorem~1.3]{BernardSchneider} and thus trivial by Theorem~\ref{theorem:exotic}. \end{proof}

\section{A question about amenability}\label{section:open.problem}

As mentioned already in Section~\ref{section:introduction}, exoticness may be viewed as a partial cause for extreme amenability: indeed, if a topological group is both exotic and amenable\footnote{A topological group $G$ is called \emph{amenable} if every continuous action of $G$ on a non-empty compact Hausdorff space admits a $G$-invariant regular Borel probability measure.}, then it is extremely amenable~\cite[Remark~5.7]{SchneiderSolecki}. Among topological unit groups of non-discrete irreducible, continuous rings, there is no unanimity concerning (extreme) amenability. Thanks to~\cite{CarderiThom}, the topological group $\A(F) = \GL(\M_{\infty}(F))$ is extremely amenable for every finite field $F$. On the other hand, as established in~\cite{SchneiderIMRN}, if $G$ is a group not \emph{inner amenable} in the sense of~\cite{effros} (e.g., a non-abelian free group), then the unit group of the ring of operators affiliated with the group von Neumann algebra of $G$ is non-amenable with respect to the relative rank topology. We amplify the following intriguing open question from the literature.

\begin{problem}[{\cite[Remark~5.12(1)]{SchneiderIMRN}}] Is the topological group $\A(\Q)$ amenable? Equivalently, is it extremely amenable? \end{problem}

\end{document}